\documentclass[12pt]{article}
\usepackage[a4paper,mag=1000,includefoot,left=2cm,right=2cm,top=2cm,bottom=2cm,headsep=1cm,footskip=1cm]{geometry}
\usepackage{amsfonts,amsmath,amssymb,amsthm}
\usepackage{bm}
\usepackage{cite}
\def\ve{\varepsilon}

\def\urho{\underline{\rho}}
\def\leq{\leqslant}
\def\geq{\geqslant}
\def\onetwelve{{\textstyle\frac{1}{12}}}
\def\half{{\textstyle\frac12}}
\def\*#1{\mathbf{#1}}
\def\nn{\nonumber}
\newtheorem{theorem}{Theorem}
\newtheorem{lemma}{Lemma}
\newtheorem{remark}{Remark}

\begin{document}
\noindent
{\large\textbf{On properties of an explicit in time fourth-order vector compact\\[2mm]
 scheme for the multidimensional wave equation}}

\bigskip\par\noindent {Alexander Zlotnik$^{a,b}$
\footnote{\small E-mail address: \text{azlotnik@hse.ru}}

\bigskip
\par\noindent $^a${\small Higher School of Economics University, Pokrovskii bd. 11, 109028 Moscow, Russia}
\par\noindent $^b${\small Keldysh Institute of Applied Mathematics, Miusskaya sqr., 4, 125047 Moscow, Russia}
}

\smallskip\par\noindent
\renewcommand{\abstract}{\normalsize ABSTRACT\\[1mm]}
\smallskip\par\noindent
\begin{abstract}
\noindent An initial-boundary value problem for the $n$-dimensional wave equation is considered.
A three-level explicit in time and conditionally stable 4th-order compact scheme constructed recently for $n=2$ and the square mesh is generalized to the case of any $n\geq 1$ and the rectangular uniform mesh.
Another approach to approximate the solution at the first time level
(not exploiting high-order derivatives of the initial functions) is suggested.
New stability bounds in the mesh energy norms and the discrete energy conservation laws are given,
and the 4th order error bound is rigorously proved.
Generalizations to the cases of the non-uniform meshes in space and time as well as of the wave equation with variable coefficients are suggested.
\end{abstract}


\medskip\par\noindent 2010 MATHEMATICS SUBJECT CLASSIFICATIONS: 65M06; 65M12; 65M15

\section{\large Introduction}
\label{s:intro}

The higher-order compact finite-difference schemes form an important class of numerical methods to solve PDEs due to their much greater computational efficiency compared to more standard 2nd order schemes.
In the last decade, the construction and study of higher-order compact schemes for initial-boundary value problems for the 2D and 3D wave or telegraph equations, including the case of variable coefficients, have attracted great interest.
Such compact schemes are implicit and most often conditionally stable.
Among them, the first type of schemes is constituted by implicit schemes which require application of FFT (for constant coefficients) or iterative methods (for variable coefficients) for their efficient implementation, see, in particular, \cite{BTT18,HLZ19,STT19,ZK20,ZChNA21}, where additional references are contained.

\par The second type of compact schemes belongs to the alternating direction implicit (ADI)-like methods (for example, see \cite{M90,S77}) which implementation is direct, reduces to solving collections of 1D linear algebraic tridiagonal systems and is much faster.
The three-level in time compact ADI schemes of several types with the approximation order $\mathcal{O}(h_t^2+|h|^4)$ were suggested for the 2D wave or telegraph equations with constant coefficients, in particular, in \cite{K10,LS13,LT10}; for the 3D and multidimensional cases, see also \cite{K10,MG13}.
Here $h_t$ and $h$ are mesh steps in time and space.
More important compact ADI schemes with the 4th approximation order, i.e. $\mathcal{O}(h_t^4+|h|^4)$, for such 2D and 3D equations were constructed and studied, in particular, in \cite{DZ13,DZ13b,Liao14,ZZLC10,ZK20}.
For the variable coefficient 2D wave equation, the compact ADI scheme having the approximation order $\mathcal{O}(h_t^2+|h|^6)$ has recently been constructed in \cite{CHP19}, and
the 4th order compact ADI-like schemes were proposed and studied in the 2D and 3D cases in \cite{LLL19,LYDH18,ZChAML21}.

\par In the recent paper \cite{JG20}, a new third type of compact schemes has been suggested for the 2D wave equation in a square, with the non-homogeneous Dirichlet boundary conditions.
The scheme is three-level explicit in time and conditionally stable; the uniform mesh in time and the square spatial mesh are taken.
Its construction begins with applying the usual three-point 1D Numerov scheme in time but then, importantly, not only the solution to the wave equation is approximated as usual but its second spatial derivatives are also approximated independently as auxiliary sought functions.
For these auxiliary functions, the three-point 1D Numerov schemes in the corresponding space directions are used.
Thus the scheme is not completely explicit (as the classical explicit 2nd order scheme) since the auxiliary functions
need to be computed by solving collections of 1D linear algebraic systems with tridiagonal matrices.
In this respect, the scheme implementation and its cost are similar to the ADI methods,
but the difference is that the auxiliary functions are completely independent and can be computed in parallel.
Theoretically, only the spectral stability analysis on harmonics was accomplished in the case of zero free terms in the equations of the scheme, but stability bounds and error bounds were not given.
The successful results of numerical experiments were presented.
Note that other two-level vector scheme with the first time derivative as the additional sought function are well-known, for example, see \cite{BB79,HLZ19,Z94,ZCh20}.

\par In this paper, the scheme from \cite{JG20} is generalized to the case of the $n$-dimensional wave equation, $n\geq 1$, and the uniform rectangular mesh (recall that the square mesh cannot be constructed in any rectangular domain), and the scheme is classified as a vector one more systematically.
The new explicit two-level vector equations for the approximate solution at the first time level are constructed too.
They are similar to the main equations and exploit only second order finite-differences of the initial functions.
Notice that we do not use derivatives of the free term and initial functions in contrast to \cite{JG20} that is essential for applications of the scheme in the non-smooth case, for example, see \cite{HLZ19,LLL19,ZK20,ZChNA21}.
The scheme is new even in the simplest 1D case.

\par Next, new results on the conditional stability of the scheme in the mesh energy norm and the discrete energy conservation law
are derived from the corresponding general results in \cite{ZK20,ZCh20} (proved by the energy method) by excluding the auxiliary sought functions.
The case of general free terms in the main and auxiliary equations of the scheme is treated that is essential for the completeness of the stability analysis.
The stability bound and the energy law contain the operators $s_{kN}^{-1}$ inverse to the Numerov averages in $k$th space direction, $1\leq k\leq n$.
The imposed stability conditions are shown to be very close to those arising in the spectral analysis, and the sufficient conditions on their validity in the standard terms of the ratios of the time and spatial steps are given.
Alternative forms of the stability bound and the energy conservation law not exploiting the operators $s_{kN}^{-1}$ are presented too.
The 4th order error bound in the mesh energy norm is also rigorously proved for the first time based on the derived stability theorem.

\par Finally, the generalizations of the scheme to the cases of the non-uniform rectangular mesh in space and the non-uniform mesh in time are presented following \cite{ZK20},
and the important case of the wave equation with variable coefficients is covered too.
For general non-uniform meshes, the scheme has the 3th approximation order as other compact schemes.
The non-uniform mesh in time is important to accomplish a dynamic choice of the time step.
The scheme becomes non-local in time in this case but it still remains explicit in time and simply implemented.
In the case of variable coefficients, the algorithm of the scheme implementation and its cost do not change, and the cost is similar to the ADI methods once again.

\par The paper is organized as follows.
Section \ref{general3level} is devoted to the construction of the scheme.
In Section \ref{stab and error bounds}, the main theorems
are proved.
Section \ref{some gener} deals with the generalizations of the scheme.

\section{\large Construction of an explicit in time 4th-order compact vector scheme for the $n$-dimensional wave equation}
\label{general3level}
\setcounter{equation}{0}
\setcounter{lemma}{0}
\setcounter{theorem}{0}
\setcounter{remark}{0}

We deal with the following initial-boundary value problem (IBVP) for the $n$-dimensional wave equation under the nonhomogeneous Dirichlet boundary condition
\begin{gather}
 \partial_t^2u(x,t)-Lu(x,t)=f(x,t),\ \ L:=a_1^2\partial_1^2+\ldots+a_n^2\partial_n^2,\ \ \text{in}\ \ Q_T=\Omega\times (0,T);
\label{hyperb2eq}\\[1mm]
 u|_{\Gamma_T}=g(x,t);\ \ u|_{t=0}=u_0(x),\ \ \partial_tu|_{t=0}=u_1(x),\ \ x\in\Omega=(0,X_1)\times\ldots\times (0,X_n).
\label{hyperb2ibc}
\end{gather}
Here $a_1>0,\ldots,a_n>0$ are constants (we take them different for uniformity with \cite{ZK20,ZChNA21}), $x=(x_1,\ldots,x_n)$, $\Omega=(0,X_1)\times\ldots\times(0,X_n)$, $n\geq 1$.
Also $\partial\Omega$ is the boundary of $\Omega$ and $\Gamma_T=\partial\Omega\times (0,T)$ is the lateral surface of $Q_T$.

\par We first rewrite formally the wave equation \eqref{hyperb2eq} as a system of equations, with a unique second order partial derivative in $t$ or $x_k$ in each of them:
\begin{gather}
 \partial_t^2u(x,t)-(a_1^2u_{11}(x,t)+\ldots+a_n^2u_{nn}(x,t)) =f(x,t)\ \ \text{in}\ \ Q_T,
\label{hyperb2eq u}\\[1mm]
 u_{kk}(x,t):=\partial_k^2u(x,t),\ \ 1\leq k\leq n,\ \ \text{in}\ \ Q_T,
\label{hyperb2eq ukk}
\end{gather}
where $u_{11},\ldots,u_{nn}$ are the auxiliary sought functions.
We consider smooth in $\bar{Q}_T$ solutions to the IBVP \eqref{hyperb2eq}-\eqref{hyperb2ibc}, then \eqref{hyperb2eq u} implies
\begin{gather}
 \partial_t^4u=\partial_t^2(Lu+f)=L\partial_t^2u+\partial_t^2f=L(a_1^2u_{11}+\ldots+a_n^2u_{nn}+f)+\partial_t^2f.
\label{pt2 hyperb2eq}
\end{gather}

\par Define the uniform mesh $\overline\omega_{h_t}=\{t_m=mh_t\}_{m=0}^M$ on a segment $[0,T]$, with the step $h_t=T/M>0$ and
$M\geq 2$.
Let $\omega_{h_t}=\{t_m\}_{m=1}^{M-1}$ be the internal part of $\overline\omega_{h_t}$.
Introduce the mesh averages (including the Numerov one) and difference operators in $t$
\begin{gather*}
 \bar{s}_ty=\half(\check{y}+y),\ \
 s_ty=\half(y+\hat{y}),\ \
 s_{tN}y=\tfrac{1}{12}(\check{y}+10y+\hat{y}),
\\[1mm]
 \bar{\delta}_ty=\frac{y-\check{y}}{h_t},\,\
 \delta_ty=\frac{\hat{y}-y}{h_t},\,\
 \mathring{\delta}_ty=\frac{\hat{y}-\check{y}}{2h_t},\
 \Lambda_ty=\delta_t\bar{\delta}_ty=\frac{\hat{y}-2y+\check{y}}{h_t^2}
\end{gather*}
with $y^m=y(t_m)$, $\check{y}^{m}=y^{m-1}$ and $\hat{y}^{m}=y^{m+1}$, as well as the operator of summation with the variable upper limit
\[
 I_{h_t}^my=h_t\sum_{l=1}^m y^l\ \ \text{for}\ \ 1\leq m\leq M,\ \ I_{h_t}^0y=0.
\]

\par Due to Taylor's formula in $t$, the wave equation \eqref{hyperb2eq} and formula \eqref{pt2 hyperb2eq}, we get
\begin{gather}
 \Lambda_tu=\partial_t^2u+\tfrac{1}{12}h_t^2\partial_t^4u+\mathcal{O}(h_t^4)
 =a_1^2u_{11}+\ldots+a_n^2u_{nn}+\tfrac{1}{12}h_t^2L(a_1^2u_{11}+\ldots+a_n^2u_{nn})
\nn\\[1mm]
 +f+\tfrac{1}{12}h_t^2(L+\Lambda_t)f+\mathcal{O}(h_t^4)
\ \ \text{on}\ \ \omega_{h_t},
\label{hyperb2eq approx in t}
\end{gather}
where we prefer to use $\Lambda_tf$ instead of $\partial_t^2f$ in \cite{JG20}.

\par Define the uniform mesh $\bar{\omega}_{hk}=\{x_{kl}=lh_k,0\leq l\leq N_k\}$ in $x_k$ with the step $h_k=\frac{X_k}{N_k}$.
Let $\omega_{hk}=\{x_{kl}, 1\leq l\leq N_k-1\}$ be its internal part.
Introduce the standard difference approximation to $\partial_k^2w$ and the Numerov average in $x_k$:
\begin{gather}
 (\Lambda_kw)_l:=\tfrac{1}{h_k^2}(w_{l+1}-2w_l+w_{\l-1}),\ \
 s_{kN}w_l:=\tfrac{1}{12}(w_{l-1}+10w_l+w_{l+1})=(I+\onetwelve h_k^2\Lambda_k)w_l
\label{basic oper}
\end{gather}
on $\omega_{h_k}$, where $w_l=w(x_{kl})$.

\par Define the uniform rectangular mesh $\bar{\omega}_h=\bar{\omega}_{h1}\times\ldots\bar{\omega}_{hk}$ in $\bar{\Omega}$
with $h=(h_1,\ldots,h_n)$.
Let $\omega_h={\omega}_{h1}\times\ldots{\omega}_{hk}$ and $\partial\omega_h=\bar{\omega}_h\backslash\omega_h$ be the internal part and boundary of $\bar{\omega}_h$.
Define also the meshes $\omega_{\*h}:=\omega_h\times\omega_{h_t}$ in $Q_T$ and $\partial\omega_{\*h}=\partial\omega_h\times\{t_m\}_{m=1}^M$ on $\bar{\Gamma}_T$, where $\*h=(h,h_t)$.

\par Let $H_h$ be the Euclidean space of functions defined on $\bar{\omega}_h$ and equal 0 on $\partial\omega_h$ and endowed with the inner product
\[
 (v,w)_h=h_1\ldots h_n\sum\nolimits_{x_{\*k}\in\omega_h}v(x_{\*k})w(x_{\*k}),\ \ \*k=(k_1,\ldots,k_n).
\]
Any operator $C_h=C_h^*>0$ in $H_h$ generates the norm $\|w\|_{C_h}=(C_hw,w)_h^{1/2}$ in $H_h$.
Recall that
\begin{gather}
 0<-\Lambda_k\leq\lambda_{\max k} I\ \ \text{with}\ \ \lambda_{\max k}=\tfrac{4}{h_k^2}\sin\tfrac{\pi(N_k-1)}{2N_k}<\tfrac{4}{h_k^2},\ \
 \tfrac23 I<s_{kN}<I\ \ \text{in}\ \ H_h.
\label{Lak sNk}
\end{gather}
\par Let $L_h$ be a difference operator such that $Lw-L_hw=\mathcal{O}(|h|^2)$ on $\omega_h$ for sufficiently smooth $w$ in $\bar{\Omega}$. Various $L_h$ can be used but below we confine ourselves by the simplest operator
\begin{gather}
 L_h:=a_1^2\Lambda_1+\ldots+a_n^2\Lambda_n\ \ \text{such that}\ \
 0<-L_h<4\big(\tfrac{a_1^2}{h_1^2}+\ldots+\tfrac{a_n^2}{h_n^2}\big)I\ \ \text{in}\ \ H_h.
\label{L_h}
\end{gather}
Then from expansion \eqref{hyperb2eq approx in t} we immediately get
\begin{gather}
 \Lambda_tu=\big(I+\tfrac{1}{12}h_t^2L_h\big)(a_1^2u_{11}+\ldots+a_n^2u_{nn})+f_{\*h}+\mathcal{O}(|\*h|^4)
 \,\ \text{on}\,\ \omega_{\*h},
\label{appr error 1}
\end{gather}
where $I$ is the identity operator and
\[
 f_{\*h}:=f+\tfrac{1}{12}h_t^2(L_h+\Lambda_t)f=(s_{tN}+\tfrac{1}{12}h_t^2L_h)f.
\]
Using the Numerov approximation in $x_k$ for equation \eqref{hyperb2eq ukk} treated as an ordinary differential equation (ODE)
in $x_k$, we have
\begin{gather}
 s_{kN}u_{kk}-\Lambda_ku=\mathcal{O}(h_k^4)\ \ \text{on}\ \ \omega_{\*h},\ \ 1\leq k\leq n.
\label{appr error 2}
\end{gather}

\par Omitting the residual terms in the last two expansions \eqref{appr error 1}-\eqref{appr error 2}, we pass to \textit{the three-level explicit in time vector compact scheme} for the wave equation
\begin{gather}
 \Lambda_tv=\big(I+\tfrac{1}{12}h_t^2L_h\big)(a_1^2v_{11}+\ldots+a_n^2v_{nn})+f_{\*h}
 \ \ \text{on}\ \ \omega_{\*h},
\label{fds 1}\\[1mm]
 s_{kN}v_{kk}=\Lambda_kv\ \ \text{on}\ \ \omega_{\*h},\ \ 1\leq k\leq n.
\label{fds 2}
\end{gather}
Here the sought functions $v\approx u$ and $v_{11}\approx u_{11},\ldots,v_{nn}\approx u_{nn}$ are defined on $\bar{\omega}_h\times\overline\omega_{h_t}$ and $\bar{\omega}_h\times\omega_{h_t}$, respectively.
We can set the discrete boundary conditions
\begin{gather}
 v|_{\partial\omega_{\*h}}=g,\ \
 a_k^2v_{kk}|_{\partial\omega_{\*h}}=g_k,\ \ 1\leq k\leq n,
\label{fds bc}
\end{gather}
where in accordance with the wave equation in $\bar{Q}_T$ and the boundary condition $u|_{\Gamma_T}=g$ one can calculate $\partial_k^2u$ on $\Gamma_T$ and thus set
\[
 g_k:=\partial_t^2g-\sum_{1\leq l\leq n,\,l\neq k}a_l^2\partial_l^2g-f\,\ \text{for}\,\ x_k=0,X_k,\ \
 g_k:=a_k^2\partial_k^2g\,\ \text{for}\,\ x_l=0,X_l,\,\ 1\leq l\leq n,\ l\neq k,
\]
on $\Gamma_T$.
Recall that equation \eqref{fds 2} is the usual Numerov scheme for the ODE \eqref{hyperb2eq ukk} in $x_k$.

\par The constructed scheme is explicit in time.
The reason is that $\hat{v}$ is found explicitly on $\omega_h$ from \eqref{fds 1}, since $\Lambda_tv=d$ means that $\hat{v}=2v-\check{v}+h_t^2d$, provided that the functions
$v_{11},\ldots,v_{nn}$ are known on $\bar{\omega}_h$.
The scheme is not completely explicit since $v_{11},\ldots,v_{nn}$
are (easily) computed from 1D difference equations \eqref{fds 2} with the boundary conditions from \eqref{fds bc}
by solving the collections of corresponding linear algebraic systems with tridiagonal matrices.
In this respect, the implementation is similar to the ADI methods.
But the difference is that the functions $v_{11}^m,\ldots,v_{nn}^m$ are independent and can be computed not sequentially but \textit{in parallel} when $v^m$ is known.
Moreover, there is no need to store all of them simultaneously since only their weighted sum is used to compute $v^{m+1}$, for $m=0,\ldots,M-1$.
\par We need also to define $v^1$ at the first time level $t=h_t$ with the 4th order of accuracy.
This can be done explicitly by using Taylor's formula and the wave equation, for example, see \cite{BTT18,DZ13b,JG20,STT19}.
But these formulas involve higher-order derivatives (or differences) of the initial functions $u_0$ and $u_1$ that is inconvenient in the non-smooth case.
We prefer to avoid this and proceed alternatively following \cite{ZK20,Z94} and construct the explicit equation for $v^1$ in the form close to \eqref{fds 1}-\eqref{fds 2}:
\begin{gather}
 (\delta_tv)^0=\tfrac12 h_t\big(I+\tfrac{1}{12}h_t^2L_h\big)\big(a_1^2v_{11}^0+\ldots+a_n^2v_{nn}^0\big)+u_{1\*h}+\tfrac{1}{2}h_tf_{\*h}^0
  \ \ \text{on}\ \ \omega_h,
\label{fds ic2}\\[1mm]
 s_{kN}v_{kk}^0=\Lambda_kv^0\ \ \text{on}\ \ \omega_h,\ \ 1\leq k\leq n,
\label{fds ic2 vkk}
\end{gather}
with suitable $u_{1\*h}\approx u_1$ and $f_{\*h}^0\approx f|_{t=0}$.
The values of $v_{kk}^0$ on $\partial\omega_h$ can be taken as in \eqref{fds bc} (or according to the values of $\partial_k^2u_0$ on $\partial\Omega$).
See additional arguments in favor of these equations for $v^1$ after the proof of Theorem \ref{theo:error bound} below.
\begin{lemma}
\label{lem1}
For the given functions
\begin{gather}
 u_{1\*h}=\big(I+\tfrac16 h_t^2L_h\big)u_1\ \ \text{on}\ \ \omega_h,
\label{u1N}\\[1mm]
 f_{\*h}^0=f_{dh_t}^{(0)}+\onetwelve h_t^2L_hf^0,\ \
 f_{dh_t}^{(0)}=f_{d}^{(0)}+\mathcal{O}(h_t^3)\ \ \text{on}\ \ \omega_h,
\label{tf0N}
\end{gather}
where $f_{d}^{(0)}:=f_0+\tfrac13 h_t(\partial_tf)_0+\onetwelve h_t^2(\partial_t^2f)_0$
with $y_0:=y|_{t=0}$ for any $y$ (see explicit forms of $f_{dh_t}^{(0)}$ in Remark \ref{rem:fdht0} below),
the following estimate for the approximation error of equation \eqref{fds ic2} holds
\begin{gather}
 \psi^0:=(\delta_tu)^0-\tfrac12 h_t\big(I+\tfrac{1}{12}h_t^2L_h\big)(a_1^2u_{11\,0}+\ldots+a_n^2u_{nn\,0})-u_{1\*h}-\tfrac{1}{2}h_tf_{\*h}^0=\mathcal{O}(|\*h|^4)
\label{appr error 3}
\end{gather}
on $\omega_h$.
\end{lemma}
\begin{proof}
Using Taylor's formula at $t=0$, the initial conditions \eqref{hyperb2ibc}, the wave equation and formula \eqref{pt2 hyperb2eq}, we get
\[
 (\delta_tu)^0=u_1+\tfrac12 h_t(Lu_0+f_0)
 +\tfrac16 h_t^2[Lu_1+(\partial_tf)_0]
 +\tfrac{1}{24}h_t^3[L(Lu_0+f_0)+(\partial_t^2f)_0]+\mathcal{O}(h_t^4).
\]
Since $Lu_0=a_1^2u_{11\,0}+\ldots+a_n^2u_{nn\,0}$, we derive
\begin{gather*}
 \psi^0=u_1+\tfrac16h_t^2Lu_1-u_{1\*h}
 +\tfrac12 h_t\big(f_{d}^{(0)}+\tfrac{1}{12}h_t^2Lf_0-f_{\*h}^0\big)
 +\mathcal{O}(h_t^4).
\end{gather*}
For functions \eqref{u1N}-\eqref{tf0N}, this implies formula \eqref{appr error 3}.
\end{proof}
\begin{remark}
If $u_0$ is smooth and known analytically, then equation \eqref{fds ic2} can be simplified as
\begin{gather}
 (\delta_tv)^0=\tfrac12 h_t\big(I+\tfrac{1}{12}h_t^2L_h\big)Lu^0+u_{1\*h}+\tfrac{1}{2}h_tf_{\*h}^0
  \ \ \text{on}\ \ \omega_h,
\label{fds ic2 simple}
\end{gather}
with omitting equations \eqref{fds ic2 vkk}, and formula \eqref{appr error 3} remains valid.
But, otherwise, equation \eqref{fds ic2 simple} cannot be used and equations \eqref{fds ic2}-\eqref{fds ic2 vkk} become preferable and are more general.
\end{remark}
\begin{remark}
\label{rem:fdht0}
Let $0<h_t\leq\bar{h}_t\leq T$.
If $f$ is sufficiently smooth in $t$ in $\bar{Q}_{\bar{h}_t}$ (or $\bar{\Omega}\times [-\bar{h}_t,\bar{h}_t]$), then $f_{dh_t}^{(0)}=f_{d}^{(0)}+\mathcal{O}(h_t^3)$ (see \eqref{tf0N}) for the following three- and two-level approximations
\begin{gather}
{f}_{dh_t}^{(0)}=\tfrac{7}{12}f^0+\half f^1-\onetwelve f^2,\ \
{f}_{dh_t}^{(0)}=\tfrac13f^0+\tfrac23f^{1/2}\ \ \text{with}\ \  f^{1/2}:=f|_{t=h_t/2}
\label{ftd02}
\end{gather}
(or
$f_{dh_t}^{(0)}=f^0+\tfrac13 h_t\mathring{\delta}_tf^0+\onetwelve h_t^2\Lambda_tf^0
 =-\onetwelve f^{-1}+\tfrac56f^0+\tfrac14 f^1$ with $f^{-1}:=f|_{t=-h_t}$).
See \cite[Remark~3]{ZK20}.
These formulas are checked easily applying Taylor's formula at $t=0$.
\end{remark}

\section{\large Stability and error bounds for the explicit 4th-order compact scheme}
\label{stab and error bounds}
\setcounter{equation}{0}
\setcounter{lemma}{0}
\setcounter{theorem}{0}
\setcounter{remark}{0}

To study stability in more detail,
we need to take the inhomogeneous version of equations \eqref{fds 2} and \eqref{fds ic2 vkk}:
\begin{gather}
 s_{kN}v_{kk}^m=\Lambda_kv^m+b_k^m\ \ \text{on}\ \ \omega_h,\ \ 0\leq m\leq M-1,\ \ 1\leq k\leq n,
\label{fds 2m}
\end{gather}
where $b_1,\ldots,b_n$ are given functions.
Note that in practice these functions are never identically zero due to the round-off errors in computations, so their influence on the solution have to be studied.
This is also necessary to derive an error bound below.

\par To state the stability theorem, we introduce the following self-adjoint operator
\begin{gather}
 A_{\*h}:=\big(I+\tfrac{1}{12}h_t^2L_h\big)\big(-a_1^2s_{1N}^{-1}\Lambda_1-\ldots-a_n^2s_{nN}^{-1}\Lambda_n\big)\ \ \text{in}\ \ H_h.
\label{A h}
\end{gather}
We impose the following condition between the steps in time and space
\begin{gather}
 \tfrac{1}{3}h_t^2\big(\tfrac{a_1^2}{h_1^2}+\ldots+\tfrac{a_n^2}{h_n^2}\big)\leq 1-\ve\ \ \text{for some}\ \ 0\leq\ve<1.
\label{stab cond 1}
\end{gather}
The second inequality \eqref{Lak sNk} implies $I<s_{1N}^{-1}<\frac32I$ in $H_h$, and thus inequalities \eqref{L_h} and this condition ensure that
\begin{gather}
 \ve I<I+\tfrac{1}{12}h_t^2L_h<I\ \Rightarrow\
 0\leq -\ve L_h<A_{\*h}<(I-\tfrac{1}{12}h_t^2 L_h)(-\tfrac32 L_h)<-\tfrac32 L_h\ \ \text{in}\ \ H_h.
\label{I p ht2Lh}
\end{gather}

\par We also define the mesh norm $\|y\|_{L_{h_t}^1(H_h)}:=\tfrac14 h_t\|y^0\|_h$.
\begin{theorem}
\label{theo:stab bound}
Let $g=0$ and $f|_{\Gamma_T}=0$.
Then, for scheme \eqref{fds 1}, \eqref{fds bc}-\eqref{fds ic2 vkk} and \eqref{fds 2m}, under the conditions \eqref{stab cond 1} and
\begin{gather}
 \tfrac14 h_t^2A_{\*h}
 \leq (1-\ve_0^2)I\ \ \text{for some}\ \ 0<\ve_0<1,
\label{stab cond 2}
\end{gather}
the following stability bound
\begin{gather}
 \max_{1\leq m\leq M}\big(\ve_0^2\|\bar{\delta}_tv^m\|_h^2+\|\bar{s}_tv^m\|_{A_{\*h}}^2\big)^{1/2}
 \leq \big(\|v^0\|_{A_{\*h}}^2+\ve_0^{-2}\|u_{1\*h}\|_h^2\big)^{1/2}+2\ve_0^{-1}\big\|f_{\*h}+\beta_{\*h}\big\|_{L_{h_t}^1(H_h)}
\label{stab bound}
\end{gather}
and the corresponding discrete energy conservation law
\begin{gather}
 \|\bar{\delta}_tv^m\|_h^2-\tfrac14h_t^2\|\bar{\delta}_tv^m\|_{A_{\*h}}^2
 +\|\bar{s}_tv^m\|_{A_{\*h}}^2
\nonumber\\[1mm]
 = \|\bar{\delta}_tv^1\|_h^2-\tfrac14h_t^2\|\bar{\delta}_tv^1\|_h^2
 +\|\bar{s}_tv^1\|_{A_{\*h}}^2+2I_{h_t}^{m-1}\big(f_{\*h}+\beta_{\*h},\mathring{\delta}_tv\big)_h
\nonumber\\[1mm]
 =\big(A_{\*h}v^0,s_tv^0\big)_h
 +\big(u_1,\delta_tv^0\big)_h
 +\half h_t\big(f_{\*h}^0+\beta_{\*h}^0,\delta_tv^0\big)_h+2I_{h_t}^{m-1}\big(f_{\*h}+\beta_{\*h},\mathring{\delta}_tv\big)_h,
\label{energy law}
\end{gather}
for $1\leq m\leq M$, are valid.
Here
\begin{gather}
\beta_{\*h}:=\big(I+\tfrac{1}{12}h_t^2L_h\big)(a_1^2s_{1N}^{-1}b_1+\ldots+a_n^2s_{nN}^{-1}b_n)\ \ \text{in}\ \ H_h\ \ \text{on}\ \ \{t_m\}_{m=0}^{M-1},
\label{fds 1 sep beta}
\end{gather}
and functions $f_{\*h},b_1,\ldots,b_n$: $\{t_m\}_{m=0}^{M-1}\to H_h$ and $v^0,u_{1\*h}\in H_h$ are any ($f_{\*h}$ and $u_{1\*h}$ are not only those specific defined above).
The first equality \eqref{energy law} holds for any $v^1\in H_h$ (not only defined by equations \eqref{fds ic2}-\eqref{fds ic2 vkk}).
\end{theorem}
\begin{proof}
1. First, we recall the general three-level scheme
\begin{gather}
 B_h\Lambda_tv+A_hv=\varphi\ \ \text{in}\ \ H_h\ \ \text{on}\ \ \omega_{h_t},
\label{fds 1 gen}\\[1mm]
 B_h(\delta_tv)^0+\tfrac12 h_t A_hv^0=u_1+\tfrac12 h_t\varphi^0\ \ \text{in}\ \ H_h
\label{fds ic2 gen}
\end{gather}
(with the weight $\sigma=0$), in particular, see \cite{Z94,ZCh20,ZChNA21,ZK20}.
Here $B_h=B_h^*>0$ and $A_h=A_h^*>0$ are any operators in $H_h$ related by the inequality
\begin{gather}
 \tfrac14 h_t^2A_h\leq (1-\ve_0^2)B_h\ \ \text{for some}\ \ 0<\ve_0<1.
\label{stab cond 2 gen}
\end{gather}
Then the following stability bound
\begin{gather}
 \max_{1\leq m\leq M}\big(\ve_0^2\|\bar{\delta}_tv^m\|_{B_h}^2+\|\bar{s}_tv^m\|_{A_h}^2\big)^{1/2}
\nn\\[1mm]
 \leq \big(\|v^0\|_{A_h}^2+\ve_0^{-2}\|B_h^{-1/2}u_{1\*h}\|_h^2\big)^{1/2}+2\ve_0^{-1}\big\|B_h^{-1/2}f\big\|_{L_{h_t}^1(H_h)}
\label{gen stab bound}
\end{gather}
and the discrete energy conservation law (that entails the stability bound)
\begin{gather}
 \|\bar{\delta}_tv^m\|_{B_h}^2-\tfrac14h_t^2\|\bar{\delta}_tv^m\|_{A_h}^2
 +\|\bar{s}_tv^m\|_{A_h}^2
\nonumber\\[1mm]
 = \|\bar{\delta}_tv^1\|_{B_h}^2-\tfrac14h_t^2\|\bar{\delta}_tv^1\|_h^2
 +\|\bar{s}_tv^1\|_{A_h}^2+2I_{h_t}^{m-1}\big(\varphi,\mathring{\delta}_tv\big)_h
\nonumber\\[1mm]
 =\big(A_hv^0,s_tv^0\big)_h
 +\big(u_1,\delta_tv^0\big)_h
 +\half h_t\big(\varphi^0,\delta_tv^0\big)_h+2I_{h_t}^{m-1}\big(\varphi,\mathring{\delta}_tv\big)_h,\ \ 1\leq m\leq M,
\label{gen energy law}
\end{gather}
are valid according to \cite[Theorem 1]{ZK20} and the proof of Theorem 1 in \cite{ZCh20} (see also \cite{S77}).
Notice that the first and second equalities \eqref{gen energy law} are valid respectively for any $v^1\in H_h$ and $v^1$ defined by equation \eqref{fds ic2 gen}.
Also inequality \eqref{stab cond 2 gen} ensures that
\[
 \ve_0^2\|\bar{\delta}_tv^m\|_{B_h}^2\leq\|\bar{\delta}_tv^m\|_{B_h}^2-\tfrac14h_t^2\|\bar{\delta}_tv^m\|_{A_h}^2\ \
 \forall v^{m-1},v^m\in H_h.
\]
\par 2. Since $g=0$ and $f|_{\Gamma_T}=0$, clearly
$v$: $\overline\omega_{h_t}\to H_h$ and $v_{kk}$: $\{t_m\}_{m=0}^{M-1}\to H_h$, $1\leq k\leq n$.
From equation \eqref{fds 2m}, we can express $v_{kk}$ through $v$:
\[
 v_{kk}=s_{kN}^{-1}(\Lambda_kv+b_k)\ \ \text{in}\ \ H_h\ \ \text{on}\ \ \{t_m\}_{m=0}^{M-1},\ \ 1\leq k\leq n.
\]
Inserting these formulas to equations \eqref{fds 1} and \eqref{fds ic2}, we get the closed equations for $v$:
\begin{gather}
  \Lambda_tv=\big(I+\tfrac{1}{12}h_t^2L_h\big)\big(a_1^2s_{1N}^{-1}\Lambda_1+\ldots+a_n^2s_{nN}^{-1}\Lambda_n\big)v+f_{\*h}+\beta_{\*h}
  \ \ \text{in}\ \ H_h\ \ \text{on}\ \ \omega_{h_t},
\label{fds 1 sep}\\[1mm]
 (\delta_tv)^0=\tfrac12 h_t\big(I+\tfrac{1}{12}h_t^2L_h\big)\big(a_1^2s_{1N}^{-1}\Lambda_1+\ldots+a_n^2s_{nN}^{-1}\Lambda_n\big)v^0
 +u_{1\*h}+\tfrac{1}{2}h_t(f_{\*h}^0+\beta_{\*h}^0)\ \ \text{in}\ \ H_h,
\label{fds 2 sep}
\end{gather}
with $\beta_{\*h}$ defined in \eqref{fds 1 sep beta}.

\par These equations present a particular case of equations \eqref{fds 1 gen}-\eqref{fds ic2 gen} with the operators
$B_h:=I$ and $A_h:=A_{\*h}$ given in \eqref{A h}.
Then the stability condition \eqref{stab cond 2 gen} takes the form \eqref{stab cond 2}, and
the stability bound \eqref{stab bound} and energy conservation law \eqref{energy law} follow from respective general relations
\eqref{gen stab bound} and \eqref{gen energy law}.
In addition, the following inequality holds 
\[
 \ve_0^2\|\bar{\delta}_tv^m\|_h^2
 \leq \|\bar{\delta}_tv^m\|_h^2-\tfrac14h_t^2\|\bar{\delta}_tv^m\|_{A_{\*h}}^2
\]
that is essential on the left in \eqref{energy law}.
\end{proof}

\par Note that the norms of $f_{\*h}+\beta_{\*h}$ and $u_{1\*h}$ in \eqref{stab bound} are bounded uniformly in $\*h$ by the same norms of $f$, $b_1,\ldots,b_n$ and $u_1$ due to the stability condition \eqref{stab cond 1}.

\par Let us compare the above stability conditions
with the corresponding conditions arising in the frequently used spectral method (it was used in \cite{JG20} as well).
In this method, one can consider the system of eigenvectors of the operator $-L_h$ in $H_h$:
\[
 -L_he_{\ell}=\lambda_{\ell}(-L_h)e_{\ell}\ \text{on}\ \omega_h,\ \ e_{\ell}=\sin\tfrac{\pi l_1 x_1}{X_1}\ldots\sin\tfrac{\pi l_n x_n}{X_n},\ \
 1\leq l_1\leq N_1-1,\ldots,1\leq l_n\leq N_n-1.
\]
with $\ell=(l_1,\ldots,l_n)$.
Hereafter $\lambda_{\ell}(A_h)$ are the corresponding eigenvalues of the operator $A_h$.
Inserting the solutions in the form $v^m=e_{\ell}y^m$, $m\geq 0$ into the homogeneous equation \eqref{fds 1 sep} with $f_{\*h}+\beta_{\*h}=0$, we get the difference equation for $y$:
\[
 \Lambda_ty+\lambda_{\ell}y=0\ \Leftrightarrow\ \hat{y}-(2-h_t^2\lambda_{\ell})y+\check{y}=0\ \ \text{on}\ \ \omega_{h_t},
\]
where
\[
 A_{\*h}e_{\ell}=\lambda_{\ell}e_{\ell}\ \ \text{on}\ \ \omega_h,\ \
 \lambda_{\ell}=\big[1-\onetwelve h_t^2\lambda_{\ell}(-L_h)\big]
 \lambda_\ell\big(-a_1^2s_{1N}^{-1}\Lambda_1-\ldots-a_n^2s_{nN}^{-1}\Lambda_n\big).
\]
The corresponding characteristic equation is
\[
 q^2-(2-h_t^2\lambda_{\ell})q+1=0,
\]
and the inequality $|q_{1,2}|\leq 1$ for its roots is the standard spectral stability condition.
But we prefer to exclude the cases $q_1=q_2=\pm 1$ since they allows $y^m$ to be unbounded in $m$ in contrast to the stability bound \eqref{stab bound}.
This leads to the condition $(2-h_t^2\lambda_{\ell})^2-4<0$ for the discriminant of the quadratic equation, i.e.,
the condition $0<\frac14 h_t^2\lambda_{\ell}<1$ for any $\ell$ that is clearly equivalent to the operator inequalities
\[
 0<I+\tfrac{1}{12}h_t^2L_h,\ \ \tfrac14 h_t^2A_{\*h}<I\ \ \text{in}\ \ H_h.
\]
They are respectively equivalent to the left inequality \eqref{I p ht2Lh} for $\ve=0$ and close to condition \eqref{stab cond 2}.
\par But the last operator inequalities ensure only a weakened form of stability.
Due to the second inequality \eqref{I p ht2Lh},
the norm $\|\cdot\|_{A_{\*h}}$ can be bounded from above by $\|\cdot\|_{-L_h}$ and from below as well provided that $\ve>0$;
so it is desirable that $0<\ve<1$.
Also the case $0<\ve_0<1$ is essential to ensure stability in the full mesh energy norm.
\par The second inequality \eqref{I p ht2Lh} implies that condition \eqref{stab cond 2} follows from the number inequality
\begin{gather}
 \tfrac14 h_t^2(1-\tfrac{1}{12}h_t^2\lambda)\tfrac32\lambda\leq 1-\ve_0^2\ \ \forall\lambda=\lambda_{\ell}(-L_h).
\label{A_h suff}
\end{gather}
This inequality follows from the quadratic inequality
\begin{gather}
 \tfrac32\big(1-\tfrac{p}{3}\big)p\leq 1-\ve_0^2\ \ \forall 0<p\leq h_t^2\big(\tfrac{a_1^2}{h_1^2}+\ldots+\tfrac{a_n^2}{h_n^2}\big).
\label{A_h suff p}
\end{gather}
Solving it, we obtain that both conditions \eqref{stab cond 1} and \eqref{stab cond 2} are valid provided that
\[
 h_t^2\big(\tfrac{a_1^2}{h_1^2}+\ldots+\tfrac{a_n^2}{h_n^2}\big)
 \leq\min\big\{3(1-\ve),\tfrac12\big[3-(1+8\ve_0^2)^{1/2}\big]\big\};
\]
note that here $\tfrac23(1-\ve_0^2)<\tfrac12\big[3-(1+8\ve_0^2)^{1/2}\big]$.
\begin{remark}
Inequality \eqref{A_h suff} is not precise since it used the inequality $s_{kN}^{-1}\leq\tfrac32 I$ in $H_h$, $1\leq k\leq n$.
Involving eigenvalues of $s_{kN}$ leads to the multivariate inequality
\[
 \big[1-\tfrac{1}{3}(p_1\sigma_1+\ldots+p_n\sigma_n)\big]
 \Big(\tfrac{p_1\sigma_1}{1-\frac{1}{3}\sigma_1}+\ldots+\tfrac{p_n\sigma_n}{1-\frac{1}{3}\sigma_n}\Big)\leq 1-\ve_0^2\ \
 \forall\sigma_1,\ldots,\sigma_n\in [0,1],
\]
with $p_1=h_t^2\tfrac{a_1^2}{h_1^2},\ldots,p_n=h_t^2\tfrac{a_n^2}{h_n^2}$, more accurate than \eqref{A_h suff p}.
Its solving is possible but rather cumbersome even for $n=1$, and here we will not come into details (see also \cite {JG20} for $p_1=p_2$, $\ve_0=0$ and $n=2$).
\end{remark}

\par Equations \eqref{fds 1 sep} and \eqref{fds 2 sep 2} that appeared in the last proof together with \eqref{fds 1 sep beta} can be written without inverse operators which is more standard for other compact schemes.
Applying the operator $\bar{s}_N:=s_{1N}\ldots s_{nN}$ to them, we obtain
\begin{gather}
  \bar{s}_N\Lambda_tv+\big(I+\tfrac{1}{12}h_t^2L_h\big)\bar{A}_Nv=\bar{s}_Nf_{\*h}+\bar{s}_N\beta_{\*h}
  \ \ \text{in}\ \ H_h\ \ \text{on}\ \ \omega_{h_t},
\label{fds 1 sep 2}\\[1mm]
 \bar{s}_N(\delta_tv)^0+\tfrac12 h_t\big(I+\tfrac{1}{12}h_t^2L_h\big)\bar{A}_Nv^0
 =\bar{s}_Nu_{1\*h}+\tfrac{1}{2}h_t\big(\bar{s}_Nf_{\*h}^0+\bar{s}_N\beta_{\*h}^0\big)\ \ \text{in}\ \ H_h,
\label{fds 2 sep 2}\\[1mm]
 \bar{s}_N\beta_{\*h}:=\big(I+\tfrac{1}{12}h_t^2L_h\big)(\bar{s}_{\widehat{1}}b_1+\ldots+\bar{s}_{\widehat{n}}b_n)\ \ \text{in}\ \ H_h\ \ \text{on}\ \ \{t_m\}_{m=0}^{M-1}
\label{fds 1 sep beta 2}
\end{gather}
since $L_h$ commutes with $\bar{s}_N$.
Here we use the operators
\[
 \bar{A}_N:=-\big(a_1^2\bar{s}_{N\widehat{1}}\Lambda_1+\ldots+a_n^2\bar{s}_{N\widehat{n}}\Lambda_n\big),\ \
 \bar{s}_{N\widehat{l}}:=\prod_{1\leq k\leq n,\,k\neq l}s_{kN}
\]
such that $\bar{A}_N^*=\bar{A}_N>0$ and $0<\bar{s}_{N\widehat{l}}^{\,*}=\bar{s}_{N\widehat{l}}<I$ in $H_h$, $1\leq l\leq n$.
According to relations \eqref{gen stab bound}-\eqref{gen energy law}, these equations imply other forms of the stability bound
\begin{gather}
 \max_{1\leq m\leq M}\big(\ve_0^2\|\bar{\delta}_tv^m\|_{\bar{s}_N}^2+\|\bar{s}_tv^m\|_{\widetilde{A}_N}^2\big)^{1/2}
\nn\\[1mm]
 \leq \big(\|v^0\|_{\widetilde{A}_N}^2+\ve_0^{-2}\|u_{1\*h}\|_{\bar{s}_N}^2\big)^{1/2}
 +2\ve_0^{-1}\big\|\bar{s}_N^{\,1/2}(f_{\*h}+\beta_{\*h})\big\|_{L_{h_t}^1(H_h)}
\label{stab bound 2}
\end{gather}
and the corresponding discrete energy conservation law
\begin{gather}
 \|\bar{\delta}_tv^m\|_{\bar{s}_N}^2-\tfrac14h_t^2\|\bar{\delta}_tv^m\|_{\widetilde{A}_N}^2
 +\|\bar{s}_tv^m\|_{\widetilde{A}_N}^2
 =\big(\widetilde{A}_Nv^0,s_tv^0\big)_h
 +\big(\bar{s}_Nu_{1\*h},\delta_tv^0\big)_h
\nonumber\\[1mm]
 +\half h_t\big(\bar{s}_N(f_{\*h}^0+\beta_{\*h}^0),\delta_tv^0\big)_h
 +2I_{h_t}^{m-1}\big(\bar{s}_N(f_{\*h}+\beta_{\*h}),\mathring{\delta}_tv\big)_h,
\end{gather}
for $1\leq m\leq M$.
Here the operators $\bar{s}_N$ and $\widetilde{A}_N$ satisfy the inequalities
\[
 (\tfrac23)^nI<\bar{s}_N<I,\ \
 \ve(\tfrac23)^{n-1}(-L_h)<\ve\bar{A}_N<\widetilde{A}_N:=\big(I+\tfrac{1}{12}h_t^2L_h\big)\bar{A}_N<\bar{A}_N<-L_h\ \ \text{in}\ \ H_h.
\]
Consequently, in the stability bound \eqref{stab bound 2}, the norms $\|\cdot\|_{\bar{s}_N}$ and $\|\cdot\|_{\widetilde{A}_N}$ can be bounded from above and below by respectively $\|\cdot\|_h$ and $\|\cdot\|_{\bar{A}_N}$ (or $\|\cdot\|_{-L_h}$) for $0<\ve<1$,
and $\bar{s}_N^{\,1/2}$ can be omitted in front of $f_{\*h}+\beta_{\*h}$.

\par Note that the scheme in the form \eqref{fds 1 sep 2}-\eqref{fds 1 sep beta 2} is not compact since it involves the product $L_h\bar{A}_N$.
The related implicit conditionally stable 4th-order compact  scheme
\begin{gather*}
  \big(\bar{s}_N+\tfrac{1}{12}h_t^2\bar{A}_N\big)\Lambda_tv+\bar{A}_Nv
  =s_Nf+\tfrac{1}{12}h_t^2\Lambda_tf,
  \ \ \text{in}\ \ H_h\ \ \text{on}\ \ \omega_{h_t},
\label{fds A N sep 2}\\[1mm]
 \big(\bar{s}_N+\tfrac{1}{12}h_t^2\bar{A}_N\big)(\delta_tv)^0+\tfrac12 h_t\bar{A}_Nv^0
 =\big(s_N+\tfrac{1}{12}h_t^2L_h\big)u_1+\tfrac{1}{2}h_t\big[f_{dh_t}^{(0)}+(s_N-I)f^0\big]\ \ \text{in}\ \ H_h,
\label{fds A_N sep 2}
\end{gather*}
together with its ADI version (which is not reproduced here) with the same properties have recently been studied in detail in \cite{ZK20}.
Here $s_N:=I+\tfrac{1}{12}h_1^2\Lambda_1f+\ldots+\tfrac{1}{12}h_n^2\Lambda_n$.

\par The next main result concerns the 4th order error bound and is based on Theorem \ref{theo:stab bound}.
Notice that $g$ and $f|_{\Gamma_T}$ can be general (not only zero) in it.
\begin{theorem}
\label{theo:error bound}
Under the stability conditions \eqref{stab cond 1} and \eqref{stab cond 2}, the following 4th order error bound in the mesh energy norm for scheme
\eqref{fds 1}-\eqref{fds ic2 vkk} with $v^0=u^0$ on $\bar{\omega}_h$ holds
\begin{gather}
 \max_{1\leq m\leq M}\big(\ve_0\|\bar{\delta}_t(u-v)^m\|_h+\sqrt{\ve}\,\|\bar{s}_t(u-v)^m\|_{-L_h}\big)=\mathcal{O}(|h|^4).
\label{error bound}
\end{gather}
\end{theorem}
\begin{proof}
We define the approximation errors of equations \eqref{fds 1}-\eqref{fds 2} as follows
\begin{gather}
 \psi:=\Lambda_tu-\big(I+\tfrac{1}{12}h_t^2L_h\big)(a_1^2u_{11}+\ldots+a_n^2u_{nn})-f_{\*h}
 \ \ \text{on}\ \ \omega_{\*h},
\label{psi order}\\[1mm]
 \psi_{kk}:=s_{kN}u_{kk}-\Lambda_ku
\ \ \text{on}\ \ \omega_{\*h},\ \ 1\leq k\leq n.
\label{psik order}
\end{gather}
Formulas \eqref{appr error 1}-\eqref{appr error 2} together with \eqref{appr error 3} mean that the approximation errors of equations \eqref{fds 1}, \eqref{fds 2} and \eqref{fds ic2}-\eqref{fds ic2 vkk} has the 4th order:
\begin{gather}
 \max_{\omega_{\*h}}|\psi|+\max_{\omega_h}|\psi^0|
 +\max_{0\leq m\leq M-1}\max_{\omega_h}|\big(|\psi_{11}^m|+\ldots+|\psi_{nn}^m|\big)
 =\mathcal{O}(|\*h|^4).
\label{appr error bound}
\end{gather}
\par Due to the equations for $v$, $v_{11},\ldots,v_{nn}$ as well as the definitions of $\psi$, $\psi_{kk}$ and $\psi^0$,
the errors $r:=u-v$, $r_{11}:=u_{11}-v_{11},\ldots,r_{nn}:=u_{nn}-v_{nn}$ satisfy the following equations
\begin{gather*}
 \Lambda_tr-\big(I+\tfrac{1}{12}h_t^2L_h\big)(a_1^2r_{11}+\ldots+a_n^2r_{nn})=\psi
 \ \ \text{on}\ \ \omega_{\*h},
\\[1mm]
s_{kN}r_{kk}-\Lambda_kr=\psi_{kk}\ \ \text{on}\ \ \omega_{\*h},\ \ 1\leq k\leq n,
\\[1mm]
 (\delta_tr)^0-\tfrac12 h_t\big(I+\tfrac{1}{12}h_t^2L_h\big)(a_1^2r_{11}^0+\ldots+a_n^2r_{nn}^0)=\psi^0\ \ \text{on}\ \ \omega_h,
\\[1mm]
 s_{kN}r_{kk}^0-\Lambda_kr^0=\psi_{kk}^0\ \ \text{on}\ \ \omega_h,\ \ 1\leq k\leq n,
\end{gather*}
with the approximation errors on the right, and
\[
 r|_{\partial\omega_{\*h}}=0,\ \ r_{kk}|_{\partial\omega_{\*h}}=0,\ \ 1\leq k\leq n,\ \ r^0=0.
\]
The stability bound \eqref{stab bound} applied to these equations and estimate \eqref{appr error bound} imply the error bound
\begin{gather*}
 \max_{1\leq m\leq M}\big(\ve_0\|\bar{\delta}_t(u-v)^m\|_h+\sqrt{\ve}\,\|\bar{s}_t(u-v)^m\|_{A_{\*h}}\big)
 \leq 2\ve_0^{-1}\big\{\|\psi^0\|_h+\tfrac32 a_{\max}^2\big(\|\psi_{11}^0\|_h+\ldots+\|\psi_{nn}^0\|_h\big)
\\[1mm]
 +I_{h_t}^{M-1}\big[\|\psi^m\|_h
 +\tfrac32 a_{\max}^2\big(\|\psi_{11}^m\|_h+\ldots+\|\psi_{nn}^m\|_h\big)\big]\big\}
 =\mathcal{O}(|\*h|^4)
\end{gather*}
with $a_{\max}:=\max_{1\leq i\leq n}$.
It suffices to apply on the left the second inequality \eqref{I p ht2Lh}.
\end{proof}

\par In the case where $v^1$ is defined alternatively to \eqref{fds ic2}-\eqref{fds ic2 vkk}, according to the first equality in the energy conservation law valid for any $v^1$, one should guarantee the 4th order approximation of $v^1\approx u(\cdot,h_t)$ in the more complicated form
\[
 \|\bar{\delta}_t(u-v)^m\|_h+\|\bar{s}_t(u-v)^m\|_{A_{\*h}}=\mathcal{O}(|\*h|^4)\ \ \text{for}\ \ m=1
\]
to prove error bound like \eqref{error bound}, cf. \cite{DZ13b}.

\par From bound \eqref{error bound}, we see that it is important to take $\ve>0$ and $\ve_0>0$ in stability conditions \eqref{stab cond 1} and \eqref{stab cond 2} to ensure the error bound in the full mesh energy norm.

\section{\large Some generalizations of the explicit in time compact vector scheme}
\label{some gener}
\setcounter{equation}{0}
\setcounter{lemma}{0}
\setcounter{theorem}{0}
\setcounter{remark}{0}

\medskip\par\noindent \textit{\textbf{\ref{some gener}.1. The case of the non-uniform mesh in space.}}

\smallskip\par\noindent Define the general non-uniform mesh
$\bar{\omega}_{hk}$ in $x_k$ with the nodes $0=x_{k0}<x_{k1}<\ldots<x_{kN_k}=X_k$
and the steps
$h_{kl}=x_{kl}-x_{k(l-1)}$, for $l=1,\ldots,N_k$ and $1\leq k\leq n$.
Let
\begin{gather*}
 \omega_{hk}=\{x_{kl}\}_{l=1}^{N_k-1},\ \
 h_{*kl}=\half(h_{kl}+h_{k(l+1)}),\ \
 h_{k\max}:=\max_{1\leq l\leq N_k} h_{kl},\ \
 h_{\max}=\max_{1\leq k\leq n}h_{k\max}.
\end{gather*}
Let now
$\bar{\omega}_{h}=\bar{\omega}_{h1}\times\ldots\times\bar{\omega}_{hn}$,
$\omega_{h}=\omega_{h1}\times\ldots\times\omega_{hn}$ and $\partial\omega_h=\bar{\omega}_{h}\backslash\omega_{h}$.

\par We generalize the difference operators \eqref{basic oper} in the standard way
\begin{gather*}
 \Lambda_kw_l=\frac{1}{h_{*kl}}\Big(\frac{w_{l+1}-w_l}{h_{k(l+1)}}-\frac{w_l-w_{l-1}}{h_{kl}}\Big),\ \
 s_{kN}w_l
 =\frac{1}{12}(\alpha_{kl}w_{l-1}+10\gamma_{kl}w_l+\beta_{kl}w_{l+1}),
\end{gather*}
on $\omega_{hk}$, where $w_l=w(x_{kl})$ and $1\leq k\leq n$,
with the coefficients
\[
 \alpha_{kl}=2-\frac{h_{k(l+1)}^2}{h_{kl}h_{*kl}},\,\
 \beta_{kl}=2-\frac{h_{kl}^2}{h_{k(l+1)}h_{*kl}},\,\
 \gamma_{kl}=1+\frac{(h_{k(l+1)}-h_{kl})^2}{5h_{kl} h_{k(l+1)}},\,\ \alpha_{kl}+10\gamma_{kl}+\beta_{kl}=12.
\]
Concerning $s_{kN}$ and its approximation error, for example, see \cite{JIS84,Z15,ZK20}.
\par Scheme \eqref{fds 1}-\eqref{fds ic2 vkk} is automatically generalized to the case of the non-uniform spatial mesh $\bar{\omega}_{h}$, with these generalized $\Lambda_k$ and $s_{kN}$.
But now the approximation error orders in \eqref{appr error 1}, \eqref{appr error 2} and \eqref{appr error 3} are reduced down to respectively
$\mathcal{O}(h_t^4+h_t^2h_{\max})$, $\mathcal{O}(h_{k\max}^3)$ and $\mathcal{O}(h_t^4+h_t^2h_{\max})$.
For smoothly varying steps $h_{kl}$ in each spatial direction, here $h_{\max}$ and $h_{k\max}^3$ can be replaced with $h_{\max}^2$ and $h_{k\max}^4$, and the approximation errors increase up to the 4th order.

\par The operator $s_{kN}$ is not self-adjoint in $H_h$ and do not commute with $\Lambda_k$, therefore the above proof of Theorem \ref{theo:stab bound} (and thus Theorem \ref{theo:error bound}) is not applicable any more, and another more cumbersome stability study like in \cite{Z15} is required.
This leads to stronger conditions on $h_t$, and, moreover, there arise some unpleasant pitfalls in this respect \cite{ZChNA21}.

\medskip\par\noindent \textit{\textbf{\ref{some gener}.2. The case of the non-uniform mesh in time.}}

\smallskip\par\noindent Define also the general non-uniform mesh $\overline\omega_{h_t}$ in $t$ with the nodes
$0=t_0<t_1<\ldots<t_M=T$ and the steps $h_{tm}=t_m-t_{m-1}$.
This case is important, in particular, to accomplish a dynamic choice of the time step, and a new moment will arise to cover it.

\par Let $h_{t+,\,m}=h_{t(m+1)}$, $h_{*t}=\half(h_t+h_{t+})$ and $h_{t\max}=\max_{1\leq m\leq M}h_{tm}$.
We define the generalized difference operators in $t$ quite similar to those in space given in the previous item
\begin{gather*}
 \Lambda_ty^m=\frac{1}{h_{*tm}}\Big(\frac{y^{m+1}-y^m}{h_{t(m+1)}}-\frac{y^m-y^{m-1}}{h_{tm}}\Big),\ \
 s_{tN}y^m
 =\frac{1}{12}(\alpha_{tm}y^{m-1}+10\gamma_{tm}y^m+\beta_{tm}y^{m+1}),
\end{gather*}
for $1\leq m\leq M-1$, with the coefficients
\[
 \alpha_{tm}=2-\frac{h_{t(m+1)}^2}{h_{tm}h_{*tm}},\
 \beta_{tm}=2-\frac{h_t^2}{h_{t(m+1)}h_{*tm}},\
 \gamma_{tm}=1+\frac{(h_{t(m+1)}-h_{tm})^2}{5h_{tm} h_{t(m+1)}},\ \alpha_{tm}+10\gamma_{tm}+\beta_{tm}=12.
\]

\par We define the average in $t$ related to the linear finite elements
\begin{gather*}
 (q_ty)^m=\frac{1}{h_{*t,m}}\int_{t_{m-1}}^{t_{m+1}}y(t)e_m(t)\,dt,\,\
 e_m(t)=\tfrac{t-t_{m-1}}{h_{tm}}\ \text{on}\ [t_{m-1},t_m],\
 e_m(t)=\tfrac{t_{m+1}-t}{h_{t(m+1)}}\ \text{on}\ [t_m,t_{m+1}],
\end{gather*}
for $1\leq m\leq M-1$.
The following two expansions follow from Taylor's formula
\begin{gather}
 q_ty=y+\tfrac13(h_{t+}-h_t)\partial_ty+\onetwelve\big(h_{t+}^2-h_{t+}h_t+h_t^2\big)\partial_t^2y+q_t\rho_3(\partial_t^3y),
\label{qt 1}\\[1mm]
 q_ty=s_{tN}y+\tilde{\rho}_{3}(\partial_t^3y),
\label{qt stN 2}\\[1mm]
 |\rho_3(\partial_t^3y)|+
 |\tilde{\rho}_3(\partial_t^3y)|\leq c_1h_{*t}^3\|\partial_t^3y\|_{C[t_{m-1},t_{m+1}]}
\label{rho bounds}
\end{gather}
on $\omega_{h_t}$, for example, see  \cite{ZK20}.

\par Similarly to \cite{ZK20}, we apply the operator $q_t$ to the wave equation \eqref{hyperb2eq}.
Due to the known formula $q_t\partial_t^2u=\Lambda_tu$ and expansion \eqref{qt 1} together with bound \eqref{rho bounds}, we obtain
\begin{gather}
 \Lambda_tu=q_t\partial_t^2u
 =\partial_t^2u
 +\tfrac13(h_{t+}-h_t)\partial_t^3u
 +\onetwelve\big(h_{t+}^2-h_{t+}h_t+h_t^2\big)\partial_t^4u+\mathcal{O}(h_{t\max}^3).
\label{Lambda tu exp}
\end{gather}
The wave equation implies
\[
 \partial_t^3u=\partial_t(Lu+f)=L\partial_tu+\partial_tf,\ \
 \partial_tu=I_t\partial_t^2u+u_1=I_t(Lu+f)+u_1,
\]
where $(I_ta)(t):=\int_0^ta(\theta)\,d\theta$.
Applying these formulas together with \eqref{pt2 hyperb2eq} once again, we pass from expansion \eqref{Lambda tu exp} to the following one
\begin{gather*}
 \Lambda_tu=
 \big[I+\onetwelve\big(h_{t+}^2-h_{t+}h_t+h_t^2\big)L\big](Lu+f)
 +\tfrac13(h_{t+}-h_t)L[I_t(Lu+f)+u_1]+f_{h_t}+\mathcal{O}(h_{t\max}^3)
\end{gather*}
with
\[
 f_{h_t}:=f+\tfrac13(h_{t+}-h_t)\partial_tf+\onetwelve\big(h_{t+}^2-h_{t+}h_t+h_t^2\big)\partial_t^2f.
\]
We replace $f_{h_t}$ with simpler $s_{tN}f$ since $f_{h_t}-s_{tN}f=\mathcal{O}(h_{t\max}^3)$ due to formulas
\eqref{qt 1}-\eqref{rho bounds}.
In addition, the well-known bound for the error of the compound trapezoid formula holds
\[
 \max_{1\leq m\leq M}|(I_ty)(t_m)-I_{h_t}^m\bar{s}_ty|=\mathcal{O}(h_{t\max}^2).
\]
Consequently, using the auxiliary functions \eqref{hyperb2eq ukk}, finally we can write
\begin{gather}
 \Lambda_tu
 =\big[I+\onetwelve\big(h_{t+}^2-h_{t+}h_t+h_t^2\big)L_h\big](u_{11}+\ldots+u_{nn})
\nn\\[1mm]
 +\tfrac13(h_{t+}-h_t)L_h\big[I_{h_t}\bar{s}_t(u_{11}+\ldots+u_{nn}+f)+Lu_1\big]
 +f_{\*h}+\mathcal{O}(h_{t\max}^3+h_{t\max}|h|^2)
\label{main expansion}
\end{gather}
with
\[
 f_{\*h}:=s_{tN}f+\onetwelve\big(h_{t+}^2-h_{t+}h_t+h_t^2\big)L_hf.
\]

\par Omitting the residual term, we pass to the following scheme
\begin{gather}
 \Lambda_tv
 =\big[I+\onetwelve\big(h_{t+}^2-h_{t+}h_t+h_t^2\big)L_h\big](a_1^2v_{11}+\ldots+a_n^2v_{nn})
\nn\\[1mm]
 +\tfrac13(h_{t+}-h_t)L_h\big[I_{h_t}\bar{s}_t(a_1^2v_{11}+\ldots+a_n^2v_{nn}+f)+Lu_1\big]
 +f_{\*h}\ \ \text{on}\ \ \omega_{\*h},
\label{fds 1 nonunif in t}\\[1mm]
 s_{kN}v_{kk}\Lambda_kv\ \ \text{on}\ \ \omega_{\*h},\ \ 1\leq k\leq n,
\label{fds 2 nonunif in t}
\end{gather}
together with equations \eqref{fds ic2}-\eqref{fds ic2 vkk} where the step $h_t=h_{t1}$ has to be taken.
In contrast to the above schemes, this scheme is \textit{non-local in time} but nevertheless its implementation is easy based on the elementary recurrent formula
\[
 I_{h_t}^m\bar{s}_td=I_{h_t}^{m-1}d+h_{tm}\bar{s}_td^m,\ \ m=1,\ldots,M-1,\ \ d:=a_1^2v_{11}+\ldots+a_n^2v_{nn}+f,
\]
though one more array is required to store this function at $m$th time level.

\par For a smoothly varying time step $h_{tm}$, the approximation error order increases up to
$\mathcal{O}(h_{t\max}^4+h_{t\max}^2|h|^2)$ in \eqref{main expansion}.
Moreover, in this case, the non-uniform mesh in space can be taken similarly to the previous Item \ref{some gener}.1 ensuring the approximation error order
$\mathcal{O}(h_{t\max}^4+h_{t\max}^2h_{\max})$ in general or even
$\mathcal{O}(h_{t\max}^4+h_{t\max}^2h_{\max}^2)$ for smoothly varying mesh steps $h_{kl}$ in all spatial directions.

\medskip\par\noindent \textit{\textbf{\ref{some gener}.3. The case of the wave equation with variable coefficients}}

\smallskip\par\noindent We also consider the generalized wave equation
\begin{gather}
 \rho(x)\partial_t^2u(x,t)-Lu(x,t)=f(x,t)\ \ \text{in}\ \ Q_T,
\label{hyperb2eq gen}
\end{gather}
with the variable coefficients $0<\urho\leq\rho(x)$ and $a_1(x)>0,\ldots,a_n(x)>0$ smooth in $\bar{\Omega}$.
Since now
\[
 \rho\partial_t^4u=\partial_t^2(Lu+f)=L\big[\tfrac{1}{\rho}(Lu+f)\big]+\partial_t^2f,
\]
we have (for the uniform meshes in time and space)
\begin{gather*}
 \rho\Lambda_tu=\rho\partial_t^2u+\onetwelve h_t^2\rho\partial_t^4u+\mathcal{O}(h_t^4)
 =\big(\rho I+\onetwelve h_t^2L\big)\big(\tfrac{1}{\rho}Lu\big)+f+\onetwelve h_t^2\big(\partial_t^2f+L\tfrac{f}{\rho}\big)+\mathcal{O}(h_t^4)
\\[1mm]
 =\big(\rho I+\onetwelve h_t^2L_h\big)\big(\tfrac{a_1^2}{\rho}u_{11}+\ldots+\tfrac{a_n^2}{\rho}u_{nn}\big)+f_{\*h}+\mathcal{O}(|\*h|^4),
\end{gather*}
where
\[
 f_{\*h}:=f+\onetwelve h_t^2\big(\Lambda_tf+L_h\tfrac{f}{\rho}\big).
\]
In a standard manner, this leads us to the following generalized equation \eqref{fds 1}:
\begin{gather*}
 \rho\Lambda_tv=\big(\rho I+\tfrac{1}{12}h_t^2L_h\big)\big(\tfrac{a_1^2}{\rho}v_{11}+\ldots+\tfrac{a_n^2}{\rho}v_{nn}\big)+f_{\*h}
 \ \ \text{on}\ \ \omega_{\*h},
\end{gather*}
together with the same equation \eqref{fds 2}.
The boundary conditions \eqref{fds bc} are the same but with the slightly more general functions
\[
 g_k:=\rho\partial_t^2g-\sum_{1\leq l\leq n,\,l\neq k}a_l^2\partial_l^2g-f\ \ \text{for}\ \ x_k=0,X_k\ \ \text{on}\ \ \Gamma_T.
\]

\par It can be also checked that equation \eqref{fds ic2} is generalized as follows
\begin{gather*}
 \rho(\delta_tv)^0
 =\tfrac12 h_t\big (\rho I+\tfrac{1}{12}h_t^2L_h\big)\big(\tfrac{a_1^2}{\rho}v_{11}^0+\ldots+\tfrac{a_n^2}{\rho}v_{nn}^0\big)
 +u_{1\*h}+\tfrac{1}{2}h_tf_{\*h}^0
  \ \ \text{on}\ \ \omega_h,
\end{gather*}
with
\[
 u_{1\*h}:=\big(\rho I+\tfrac16 h_t^2L_h\big)u_1,\ \
 f_{\*h}^0:=f_{dh_t}^{(0)}+\onetwelve h_t^2L_h\tfrac{f^0}{\rho}\ \ \text{on}\ \ \omega_h,
\]
together with the same equation \eqref{fds ic2 vkk}.
These formulas ensure the following estimate for the approximation error
\begin{gather*}
 \rho(\delta_tu)^0-\tfrac12 h_t\big(\rho I+\tfrac{1}{12}h_t^2L_h\big)
 \big(\tfrac{a_1^2}{\rho}u_{11\,0}+\ldots+\tfrac{a_n^2}{\rho}u_{nn\,0}\big)-u_{1\*h}-\tfrac{1}{2}h_tf_{\*h}^0=\mathcal{O}(|\*h|^4)
\ \ \text{on}\ \ \omega_h.
\end{gather*}

\par For the derived scheme, the stability and error bounds need another study.
The scheme can be generalized for the case of non-uniform meshes in space and time similarly as described above.

\bigskip\par\noindent\textbf{Disclosure statement}
\medskip\par\noindent No potential conflict of interests was reported by the author.

\bigskip\par\noindent\textbf{Funding}
\medskip\par\noindent This paper was supported by the Russian Science Foundation under Grant no.~19-11-00169.

\bigskip\par\noindent\textbf{ORCID}
\medskip\par\noindent \textit{Alexander Zlotnik} https://orcid.org/0000-0003-2440-2816

\renewcommand{\refname}{\large\textbf{References}}

\end{document}